\documentclass[reqno,12pt,letterpaper]{amsart}
\usepackage{amsmath,amssymb,amsthm,graphicx,mathrsfs,url}
\usepackage[usenames,dvipsnames]{color}
\usepackage[colorlinks=true,linkcolor=Red,citecolor=Green]{hyperref}
\usepackage{amsxtra}
\usepackage{epstopdf}

\usepackage[dvipsnames]{xcolor}
\usepackage{bbm}
\usepackage{tikz}

\usepackage{graphicx}

\def\?[#1]{\textbf{[#1]}\marginpar{\Large{\textbf{??}}}}

\def\smallsection#1{\smallskip\noindent\textbf{#1}.}
\numberwithin{equation}{section}

\def\smallsection#1{\smallskip\noindent\textbf{#1}.}
\let\epsilon=\varepsilon 
\newcommand{\dd}{\, {\rm d}}

\newcommand{\R}{\mathbb{R}}

\usepackage{amsxtra}
\usepackage{epstopdf}

\newtheorem{theorem}{Theorem}
\newtheorem{proposition}{Proposition}[section]	
\newtheorem{definition}[proposition]{Definition}
\newtheorem{assumption}{Assumption}

\newtheorem{lemma}[proposition]{Lemma}

\newtheorem{remark}{Remark}

\title{$L^2$-stability near equilibrium for the $4$ waves kinetic equation}
\author{Angeliki Menegaki} 
\address{Institut des hautes études Scientifiques, 35 Rte de Chartres, 91440, Bures-sur-Yvette, France}
\email{menegaki@ihes.fr}

\begin{document}
\maketitle

\begin{abstract}
We consider the four waves spatial homogeneous kinetic equation arising in wave turbulence theory. We study the long-time behaviour and existence of solutions around the Rayleigh-Jeans equilibrium solutions. For cut-off'd frequencies, we show that for dispersion relations weakly perturbed around the quadratic case, the linearized operator around the Rayleigh-Jeans equilibria is coercive. We then pass to the fully nonlinear operator, showing an $L^2$ - stability for initial data close to Rayleigh-Jeans. 
\end{abstract}
\tableofcontents

\section{Introduction}
We aim to study the spectral properties of the linearized \emph{kinetic wave equation} or \emph{phonon Boltzmann equation}. According to the theory of wave turbulence, the so-called kinetic wave equation is obtained as a kinetic limit of time-reversible wave dispersive equations where many waves interact weakly non-linearly and it predicts the evolution of the average distribution of energy
in the frequency space. 
Put differently, in analogy with the kinetic theory of particles, one can directly compare the kinetic wave equation with the usual Boltzmann equation for dilute gases, obtained in the Boltzmann-Grad limit, describing the evolution of the distribution of particles in the phase space. 

This wave kinetic equation has dynamical properties similar to the Boltzmann equation as an $H$-theorem holds and the mass, momentum, and energy are conserved quantities. The system is driven towards the equilibrium solution which is the Rayleigh-Jeans distribution. In addition to this equilibrium stationary state, it was shown by Zakharov \cite{Zakharov66} that also non-equilibrium steady states exist, the so-called Kolmogorov–Zakharov solutions. These are characterised by the presence of non-zero fluxes 
 where energy is transferred through direct cascades from small to large frequencies and mass through inverse cascades from large to small frequencies.
For more details we refer to the books \cite{Zakharov98,Nazarenko,NewellRumpf}. 

The research area of wave turbulence theory has now attracted a lot of attention due to the wide range of practical applications. To mention a few the wave kinetic equation appears in the studies of plasma waves, surface capillary and gravity waves \cite{Zakharov:1999,Hasselmann:1962}, nonlinear optics \cite{Dyachenko:1992} and elastic plates \cite{During:2017}.

\subsection{Kinetic wave equation and properties}
The space homogeneous kinetic wave equation reads as follows
\begin{equation} \label{eq:NL WKE}
\begin{split}
    \partial_t n(t,k) &=\mathcal{L} n(t,k) \\ &:=  - \int_{\mathbb{R}^9} \frac{n n_1n_2n_3}{|k|^{\beta/2}|k_1|^{\beta/2} |k_2|^{\beta/2}|k_3|^{\beta/2}} \left(\frac{1}{n_1} + \frac{1}{n_2} - \frac{1}{n_3} - \frac{1}{n}\right) \delta_{\Omega} \delta_{\Sigma} \dd k_1 \dd k_2 \dd k_3
    \end{split}
\end{equation}
where $$\Sigma = \Sigma (k, k_1, k_2,k_3) = k_1+k_2-k_3-k\ \text{and}\ \Omega= \Omega (k, k_1, k_2,k_3) = \omega_1+\omega_2-\omega_3-\omega, $$ $n_i= n(k_i), \omega_i=\omega(k_i)$ and $k \mapsto \omega(k)$ is the dispersion relation. Finally $\beta$ in the cross-section is a parameter originated from the nonlinearity of the dynamics in the microscopic level and in particular the value $\beta=0$ corresponds to the cubic power law in the nonlinear Schr\"{o}dinger equation. 

From the collision operator one can easily see that the mass, momentum and energy are conserved at least formally and that the equilibrium solutions take the form for $(a,b,c) \in \R \times \R^3\times \R_+$:  $$ f_{\infty}(k)= \frac{1}{a+b\cdot k + c \omega},\quad  a+b\cdot k + c \omega >0,\ c >0.$$
as the collision integral vanishes.

\subsection{Linearization around Rayleigh-Jeans equilibria}
We linearize around the equilibrium solution $ f_{\infty}(k) = \frac{1}{\omega+1 \cdot k+1}$:  
\[n(k)= f_{\infty}(k) (1+\varepsilon g(k)).\] The new unknown $g$ satisfies the following linearized equation 
\begin{equation}
\begin{split}
    \partial_t g(t,k) &= \frac{f_{\infty}^{-1}}{\varepsilon} \partial_t n(t,k) \\ 
    &= f_{\infty}(k)^{-1} \int_{\mathbb{R}^9} \frac{ f_{\infty}(k)f_{\infty}(k_1)f_{\infty}(k_2)f_{\infty}(k_3)}{|k|^{\beta/2}|k_1|^{\beta/2} |k_2|^{\beta/2}|k_3|^{\beta/2}} \times \\ 
&\times \big[ f_{\infty}(k_1)^{-1} g_1 + f_{\infty}(k_2)^{-1} g_2 - f_{\infty}(k_3)^{-1} g_3 - f_{\infty}(k)^{-1}  g \big] 
 \delta_{\Omega} \delta_{\Sigma} \dd k_1 \dd k_2 \dd k_3+ \mathcal{O}(\varepsilon).
\end{split}
\end{equation}
This yields the linearized operator 
\begin{equation} \label{eq:linearised_op}
\begin{split} 
  L g (k) :=  &\int_{\mathbb{R}^9} \frac{ f_{\infty}(k_1)f_{\infty}(k_2)f_{\infty}(k_3)}{|k|^{\beta/2}|k_1|^{\beta/2} |k_2|^{\beta/2}|k_3|^{\beta/2}} \times \\ 
&\times \big[ f_{\infty}(k_1)^{-1} g_1 + f_{\infty}(k_2)^{-1} g_2 - f_{\infty}(k_3)^{-1} g_3 - f_{\infty}(k)^{-1}  g \big]
 \delta_{\Omega} \delta_{\Sigma} \dd k_1 \dd k_2 \dd k_3. 
 \end{split} 
 \end{equation} 
This operator is self-adjoint on $L^2$ and the Dirichlet form of $-L$ satisfies:
\begin{equation}
\begin{split}
\mathcal{D}(g) := -\langle L g,g &\rangle_{L^2}  =  \frac{1}{4} \int  \frac{f_{\infty}(k_1)f_{\infty}(k_2)f_{\infty}(k_3)}{|k|^{\beta/2}|k_1|^{\beta/2} |k_2|^{\beta/2}|k_3|^{\beta/2}} \times \\ & \times \big[ f_{\infty}(k_1)^{-1}g_1 + f_{\infty}(k_2)^{-1}  g_2 - f_{\infty}(k_3)^{-1} g_3 - f_{\infty}(k)^{-1}  g \big]^2 \delta_{\Omega} \delta_{\Sigma} \dd k_1 \dd k_2 \dd k_3 dk \\  &\geq 0
\end{split}
\end{equation}
which implies that the spectrum of $L$ lies on $\mathbb{R}_-$. Its null space is 
\[\operatorname{Ker}(L) =\operatorname{span} \{ 1, k , \omega\}. \]

\subsection{Main ideas and assumptions}

The goal is (i) to study  the linearized around the Rayleigh-Jeans equilibrium operator $L$ and to show the existence of an $L^2$ spectral gap. For this purpose we split the operator into a sum of two integral operators for which we show that one is coercive while the other is a compact operator. Then we can ensure the existence of a spectral gap due to the coercive part by Weyl's theorem. 
(ii) We control the $L^2$ norm of the nonlinearities to show through a fixed-point argument that a mild solution of $\mathcal{L}$ exists in a ball around the Rayleigh-Jeans equilibrium whose radius depends on the spectral gap of $L$. 
Finally using all the above we show the existence of a basin of attraction of Rayleigh-Jeans, characterized by an $L^2$ norm. 
 This argument is reminiscent of an argument of Grad used to study the linearized Boltzmann operator for hard potentials and it is the first time that it is used to obtain stability for the wave kinetic equation. 

For these we need the following assumptions: 

\noindent
\textbf{Cut-off for the frequencies}: Regarding the domain of the frequencies, we restrict to the domain  $[0, k_c]$ for a finite constant $k_c$ for each component, i.e. $|k|,|k_i| \leq k_c$. This set-up is not physically motivated as normally we do not expect the frequencies to stay bounded. This cut-off assumption is  however needed in order to control the singularities appearing in the collision operator. In particular to ensure the coercivity and the compactness properties of the splitting of the linearised operator. Let us stress that we impose this assumption to all the wavenumbers, so that the symmetries of the collisional operator are conserved. To picture this, see fig. \ref{fig: cutoff} on the plane when the dispersion relation is quadratic and the resonance condition is interpreted as a collision among particles leading to another group of particles.

We will assume for the dispersion relation  the following
 
\begin{assumption}[Assumptions on the dispersion relation]
  \label{ass: dispersion}
  We assume the following for the dispersion relation $\omega(k)$
  \begin{itemize}
  \item[(i)] $\omega(k)=\Omega(\vert k \vert)$, $\Omega(|k|)= |k|^2+ s(|k|)  \in C^2(\mathbb{R}_+)$. 
  \item[(ii)]  There exist $C_1, C_2, C_3$ positive finite constants with 
    \begin{equation} 
        \| s \|_{L^{\infty}(\R_+)} \leq C_1\  \text{and }\  0< C_2  \leq \frac{\Omega'(x)}{x} \leq C_3,\ x \in \R_+.
    \end{equation} 
    \item[(iii)] The Hessian of the perturbation $S(x) = s(|x|)$ is a diagonal positive definite so that $\Lambda_1 |z|^2 \leq z^T\operatorname{Hess} S(x) z \leq \Lambda_2 |z|^2$  for all $z,x\in \mathbb{R}^3$, for finite positive constants $\Lambda_1,\Lambda_2$.
  \end{itemize}
\end{assumption}
    That means that we consider a positive dispersion relation under the constraint that it has a quadratic growth while allowing for some oscillations as well. 



\subsection{State of the art}  Typically there are two types of such kinetic equations in wave turbulence, namely the $3$-waves interactions and the $4$-waves interactions. Peierls \cite{Peierls1929} in 1929 was the first to write down the $3$-wave equation termed as the phonon Boltzmann equation, where he derived it from harmonic lattices perturbed with weak anharmonicities. Later the subject was revised by Zakharov and his collaborators, see \cite{ZakLvovFalkov}, especially after the discovery of the non-equilibrium Kolmogorov–Zakharov spectra.
 
Regarding its rigorous derivation, see first a proposed program in \cite{Spohn:2006} where the starting point is, as also suggested by Peierls,  a harmonic lattice with a weakly anharmonic (non-quadratic) on-site potential, otherwise known as the FPUT chain. More recently there many important advances starting from cubic nonlinear Schr\"{o}dinger equation with random initial data, see \cite{LukSpohn},  then \cite{BGHS21,CollGermain,DengHani1} up to the kinetic time-scale (up to errors) and a derivation reaching the kinetic time-scale \cite{DengHani2}. 

Concerning now the well-posedness of \eqref{eq:NL WKE}, in \cite{EV13} under velocity-isotropy assumption and for quadratic dispersion relation, without imposing forcing the authors studied the long-time behaviour in a weighted $L^{\infty}$ space. In particular in \cite{EV13}  the authors obtained global, measure valued, weak solutions and showed that as time goes to infinity condensation can happen, where all the mass is concentrated at the origin. Furthermore they studied the transfer of energy in higher frequencies and prove a stability result around the Kolmogorov-Zakharov spectrum. 
We also mention the article \cite{SaraMerino16} where existence and uniqueness of radial weak solutions to a simplified version of the $4$-waves interactions equation in a broad range of dispersion relations was studied.

More recent works on the local well-posedness for the homogeneous $4$-waves kinetic equation for a general class of dispersion relations and without the isotropy assumption are in \cite{GIT20}, where the authors work in weighted $L^\infty$ spaces. There, restricting to the quadratic dispersion relation, they have the estimates in weighted $L^2$ as well. Moreover a study of space inhomogeneous $4$-waves interaction is done recently in \cite{Ioakeim22} where global well-posedness and preservation of sign in $L^\infty$ is proven for a perturbation of the vacuum exploiting the dispersive properties of the  transport operator.

Concerning works on the $3$-waves interactions in weak wave turbulence, it has been studied in many works for phonon interactions in anharmonic crystals which as Peierls suggested leads to the kinetic model of a nonlinear Boltzmann equation of interacting phonons and it has the same mathematical formulation as the quantum Boltzmann equation of bosons at very low temperature, see for example \cite{EscobTran2015, CraciunTran:2016,EscobPezzValle2011}. It is also studied in the context of capillary waves, \cite{NguyenTran18} and for oceanic internal gravity waves \cite{GambaTranSmith2017}. We also  mention the recent study in \cite{SofferTran2020} for acoustic waves where the authors study the energy cascades phenomenon and where a comparison with coagulation-fragmentation type of equations is considered. 
Finally we mention the article \cite{Escobedo22} where the author studied an approximation of the linearized $3$-waves kinetic equation around Rayleigh-Jeans equilibria. Existence, uniqueness and properties of the fundamental solution are obtained and the initial value problem is solved for integrable and locally bounded initial data. 

As for the more complicated and physicaly interesting case of stability around the Kolmogorov-Zakharov non-equilibrium solutions, see the preprint \cite{CDG22} where an $L^{\infty}$ stability is provided when the indirect mass cascade is taken into account.

\subsection{Statements of the results} Under the previous assumptions we first prove a coercivity-result for the linearized operator.

\begin{theorem} \label{theo:coerc linearised}
Let the hypothesis \ref{ass: dispersion} (i),(ii) hold, $k ,k_i \in [0,k_c]^3$ for $k_c < \infty$ and $\beta =0$. Then the linearized operator $L_{k_c}$ is coercive, i.e.
$$ \forall g \in  L^2( [0,k_c]^3 ) \cap  (\operatorname{Ker}(L_{k_c}))^\perp, \quad \mathcal{D}^{k_c}(g)
 \geq \lambda  \|g \|_{L^2([0,k_c]^3)}^2, $$
for some positive constant $\lambda = \lambda(k_c)$, where $ \mathcal{D}^{k_c}(g)=-\langle L_{k_c} g, g \rangle$ and $L_{k_c}$, defined in \eqref{eq: K and A operators}, is the operator $L$ in \eqref{eq:linearised_op} for the cut-off'd frequencies. 
\end{theorem}

 \begin{definition}[Mild solutions]  We say that $f=f(t,k)$ on $(0,\infty) \times \R^3$  is a mild solution to \eqref{eq:NL WKE} with initial data $f_0 \in L^2(\R^3)$ when $f \in C^1((0,\infty), L^2(\R^3,\R))$, for $t>0$ and 
 $ f(t,k) = f_0(k) + \int_{0}^{t}  \mathcal{L} f (s,k) \dd s \in L^2.$
 \end{definition} 

We then pass to the fully nonlinear operator. We deal with the question of existence of mild solutions close to equilibrium and their $L^2$ - stability. 
 
\begin{theorem}\label{theo: stability NL}
Let $C>0$, $\beta=0$, $\omega(k)=|k|^2+ s(|k|)$ satisfying assumption \ref{ass: dispersion} (i),(ii), (iii) and $k,k_i \in [0,k_c]^3$. Assuming that the linearized operator has an $L^2$ spectral gap $\lambda >0$, which is the case when  $g_t \in (\operatorname{Ker}(L_{k_c}))^{\perp}$  for all $t\geq 0$, and the initial data $g_0$ so that $\|g_0\|_{L^2([0,k_c]^3)} \leq \frac{\sqrt{\lambda}}{2\sqrt{2C}}$,  then there is a mild solution $g_t$ satisfying $\|g_t\|_{L^2([0,k_c]^3)} \leq \frac{\sqrt{\lambda}}{2\sqrt{2C}}$. 
Also, when $\Pi$ is the projection onto the equilibria
\begin{align*} 
\| \Pi^\perp g_t \|_{L^2([0,k_c]^3)}& \leq e^{ -\frac{\lambda}{2}t} \| \Pi^\perp g_0 \|_{L^2([0,k_c]^3)} + (1-e^{-\frac{\lambda}{2}t}) \| \Pi g_0 \|_{L^2([0,k_c]^3)}. 
\end{align*}
In particular as soon as $\Pi g_0=0$, $ \Pi^{\perp} g_t$ decays exponentially fast. 
\end{theorem}

\begin{remark}[On the assumptions] Even though everything is written in $d=3$, 
the results hold for all dimensions $d \geq 2$. The further assumption $\beta=0$, is needed to ensure the compactness part in the linearised operator and also for the $L^2$-estimates to control the cubic nonlinearity. These estimates are required for the global stability of the fully nonlinear operator. Finally let us comment on the two reasons we restrict that all the frequencies $|k|,|k_i| \leq k_c$, cf fig \ref{fig: cutoff}: 
\begin{enumerate} 
\item[(i)] First, on the level of the linearised operator $L$ we want $L=$ coercive part $+$ compact part. The cut-off assumption is used in order to prove the compactness property of the second part. 
We are led therefore by technicalities to assume the cut-off to ensure compactness, but we do not expect it to be a necessary assumption for our result. Removal of the cut-off in this functional setting will be an interesting extension of this work. 
\item[(ii)] Second, for the fixed-point argument on the nonlinear level, the cut-off is used to ensure  an upper bound on the equilibrium $f_{\infty}$. 
\end{enumerate} 
\end{remark}

\begin{remark} 
Notice that even though here we show that when we are close to the Rayleigh-Jeans equilibrium, the dynamics remains close to it in $L^2$, this is not in general the case even when there is no external forcing in the system. An example of a different long-time behaviour in the isotropic case under finite mass is in \cite{EV13}, where weak stationary solutions are characterised as Dirac measures. 
\end{remark}

\begin{remark}
Here we stress that this stability property holds only for solutions $g_t \in \operatorname{Ker}(L_{k_c})^{\perp}$ for all $t \geq 0$, meaning that we look rather at the evolution of $\Pi^{\perp} g_t $ where $\Pi$ is the projection onto the equilibria (see also the discussion in the beginning of subsection \ref{subsect: existence}). Equivalently, initially we impose a mean-zero property, both for the mass and the energy: at $t=0$, $ \int g_0(k) f_{\infty}(k) \dd k=  \int g_0(k)\omega(k) f_{\infty}(k) \dd k=0$,  which due to the conservation of mass and energy also holds for all later times, at least formally. This way we do not allow mass or energy fluctuations which makes our set-up here different than the one studied in \cite{EscobValle}. 
\end{remark}

\begin{figure} \label{fig: cutoff}
  \centering
  \begin{tikzpicture}
    \draw[-latex] (-3.6,0) -- (3.6,0) node[right] {$$};
    \draw[-latex] (0,-3.6) -- (0,3.6) node[above] {$$};

    \draw (0,0) circle [radius=2.9];
    \node at (3.1,0) [fill=white, inner sep=1pt] {$k_c$}; 
    \draw[rotate=45] (1.1,0.0) rectangle ++(1, 1.5);
\node at (0.7,0.6) {$k_3$};
\node at (0.6,2.7) {$k$};
\node at (1.7,1.5) {$k_1$};
\node at (-0.5,1.9) {$k_2$};


  \draw[rotate=40, dashed] (1.0,-2.5) rectangle ++(1, 1.5);
  \end{tikzpicture}
  \caption{In $2$-dim, as long as $|k|,|k_i|\leq k_c$ for all $i=1,2,3$ we allow collisions. The dashed rectangular corresponding to a collision when the resulting wavenumbers are outside of the domain is not allowed in our set-up.}
\end{figure}
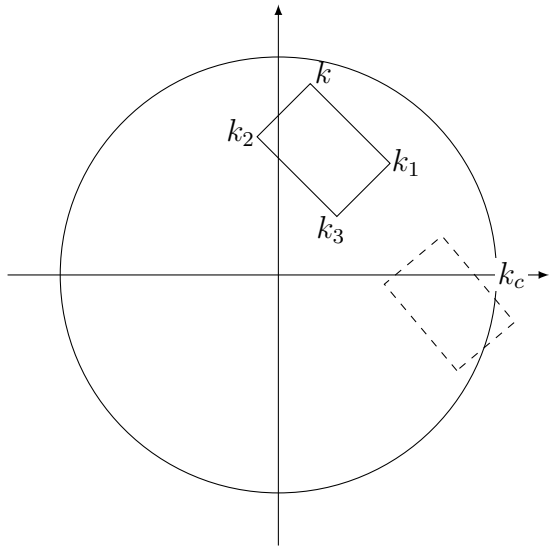

\smallsection{Notation} 
For a parameter $s\geq 0$, the space $L^2_s:=  \{ g : \| \langle k \rangle^{s}  g \|_{L^2} < \infty \} $ where $\langle  k \rangle = (1+|k|^2)^{1/2}$. 
We also occasionally write $A \lesssim_{\lambda} B$ in order to say that $A \leq C B$ for some constant $C = C(\lambda)$ that depends on $\lambda$.  We also employ the Kronecker delta where $\delta_{k \in I} = 1$ if $k\in I$ and $0$ otherwise. The inner product of two vectors $x,y \in \R^d$ is denoted either by $\langle x,y \rangle$ or $x \cdot y$.

\section{Method of proofs} 
First we present the parametrisation of the resonant manifold on which the integral collision operator acts. This parametrisation for more general dispersion relations  can be found in \cite{GIT20} (also in \cite{NT19} for the quantum Boltzmann equation). 

\subsection{Parametrization of the resonant manifold} 
When the dispersion relation is the quadratic $\omega(k)=|k|^2$ and the solution $n$ is radial, the equation takes the form of the one-dimensional Boltzmann equation 
\begin{align*} 
 \partial_t &f(t,|k|^2) = \int_{\mathbb{R}_+^2}\frac{|k_2| |k_1| \text{min}(|k|,|k_1|,|k_2|,|k_3|)}{|k| (|k||k_1||k_2||k_3|)^{\beta/2}}\times \\ 
 & \times  \Big( f(|k_2|^2)f(|k_1|^2) \big[f(|k|^2)+f(|k_3|^2) \big]-f(|k|^2)f(|k_3|^2) \big[f(|k_2|^2)+f(|k_1|^2)\big]  \Big) \dd |k_1| \dd |k_2|. 
 \end{align*}
where $|k_3|^2= |k_1|^2+ |k_2|^2- |k|^2$. Since the conservation laws known for Boltzmann hold for this isotropic model as well, the Carleman’s collision parametrization \cite{LaureSR09, Villani_review, cercignani1988} can be applied in order to write the collision operator as integrals over spheres. This is helpful as then its decomposed parts have exact and convenient forms making it easier to prove the desired coercivity property for the one part and the compactness property for the rest.  


Considering more general dispersion laws for the nonlinear wave kinetic equation \eqref{eq:NL WKE}, we set $z=k_2$ and let $ S_{k,k_3}$ be the resonant manifold, i.e. the set of zeros of $\mathcal{C}(z) = \omega(k_3-z+k) - \omega(k) - \omega(k_3)+ \omega(z)$.
Then 
 $$\delta_{\omega_1+\omega_2-\omega_3-\omega}  \dd z \dd k_3 = \frac{\dd \mu(z)}{ |\nabla \mathcal{C}(z) | } \dd k_3 $$
where $\mu$ is the surface measure on $S_{k,k_3}$. The derivative of $\mathcal{C}$ is given by 
$$ \nabla_z \mathcal{C} = \frac{z-k-k_3}{|z-k-k_3|} \Omega'(|k+k_3-z|)+ \frac{z}{|z|}\Omega'(|z|)$$
and the directional derivative in any direction vertical to $k+k_3$ is positive, i.e.  $\mathcal{C}(z)$ is strictly increasing in any direction orthogonal to $k+k_3$. Indeed if $q$ is a vector orthogonal to $k+k_3$,
$$ \langle q , \nabla_z \mathcal{C} \rangle = |q|^2 \left( \frac{\Omega'(|z|)}{|z|}+ \frac{\Omega'(|k+k_3-z|)}{|k+k_3-z| }\right) > 0. $$

For a parameter $\alpha \in \mathbb{R}$, we write the vector $z$ as follows: $z= \alpha (k+k_3) + q$ where $\langle q,(k+k_3) \rangle =0$. Then since $\mathcal{C}(z)= \mathcal{C}(\alpha (k+k_3) + q)$ is a radial function in $q$
we obtain that the intersection of $S_{k,k_3}$ with the plane 
$$P_\alpha:=\{\alpha (k+k_3) + q: \ \langle q,(k+k_3) \rangle = 0\}, $$
if it is not empty, it is a circle with radius  $r_\alpha< \infty$ and centre at $\alpha (k+k_3)$.

We then consider a vector $(k+k_3)^{\perp}$ in the plane $P_0:=\{q : \langle q,k+k_3 \rangle = 0\}$ and $e_\theta  \in \R^3$ a unit vector in $P_0$ so that the angle between $(k+k_3)^{\perp}$ and $e_\theta $ is $\theta$. We can then parametrise the resonant manifold as
$$
S_{k,k_3} = \Big\{z(\alpha, \theta)=\alpha (k+k_3) +r_\alpha e_\theta : \theta \in [0,2 \pi], \alpha  \in  \mathcal{A}_{k,k_3} \Big\}
$$ 
where $\mathcal{A}_{k,k_3}$ is the set of $\alpha \in \R$ so that there is a solution to $\mathcal{C}(z)=0.$

We note that we can interpret $\mathcal{C}$ as a function of $\mathcal{C}(r,\alpha)$ where $\partial_r \mathcal{C}(r,\alpha) >0$ and thus through the implicit function theorem we can write $r=r_\alpha$ for $\alpha \in  \mathcal{A}_{k,k_3} $ via  a smooth function that is vanishing on the boundary. We extend $\alpha \in \R$ by setting $r_\alpha=0$ outside of the set $\mathcal{A}_{k,k_3}$. 

For the computation of the surface area we write $$\dd\mu(z) = |  \partial_\alpha z \times \partial_\theta z|\ \dd\alpha\ \dd\theta =   | ( (k+k_3) + \partial_\alpha r_\alpha e_\theta ) \times r_\alpha \partial_\theta e_\theta |\ \dd\alpha\ \dd \theta. $$
Since $\partial_\theta e_\theta$ is orthogonal to $e_\theta$ and to $k+k_3$, we write 
\begin{equation} \label{eq: parametr1}
\begin{split}
 \dd\mu(z) &= |r_\alpha|  \big\vert  ( (k+k_3) + \partial_\alpha r_\alpha e_\theta )  \big\vert \dd\alpha\ \dd\theta 
  = \left(  r_\alpha^2 | k+k_3|^2 + \frac{|\partial_\alpha (r_\alpha^2) |^2}{4} \right)^{1/2} \dd\alpha\ \dd\theta. 
\end{split}
\end{equation}
Now  in order to make it explicit we want to compute $\partial_\alpha (r_\alpha^2)$. This is written in terms of $\partial_\alpha  |z_\alpha |^2$, for $z_\alpha = \alpha (k+k_3) + q$, as $ |z_\alpha |^2 =\alpha^2 | k+k_3 |^2 + r_\alpha^2$ which implies $\partial_\alpha (r_\alpha^2) =  \partial_\alpha  |z_\alpha |^2 - 2\alpha | k+k_3 |^2.$

We differentiate the relation $\mathcal{C}(z_{\alpha})=0$ and then we take the inner product with $\partial_\alpha z_{\alpha}$ to have 
\begin{align*}
&0=\left\langle  \partial_\alpha z_\alpha , \nabla_z \mathcal{C}(z_\alpha) \right\rangle \\
&=  \left\langle
\partial_\alpha z_\alpha , z_\alpha \left(\dfrac{\Omega'(|k+k_3-z_\alpha|)}{|k+k_3-z_\alpha|} + \dfrac{\Omega'(|z_\alpha|)}{|z_\alpha| } \right)
\right\rangle  - \left\langle
\partial_\alpha z_\alpha , \dfrac{k+k_3}{|k+k_3-z_\alpha|} \Omega'(|k+k_3-z_\alpha|)
\right\rangle
\end{align*}
so that 
\begin{equation}
\begin{split}
\partial_\alpha |z_\alpha|^2 &= 
2 \frac{|k+k_3|^2}{\dfrac{\Omega'(|k+k_3-z_\alpha|)}{|k+k_3-z_\alpha|} + \dfrac{\Omega'(|z_\alpha|)}{|z_\alpha| } }\dfrac{\Omega'(|k+k_3-z_\alpha|)}{|k+k_3-z_\alpha|}.
\end{split}
\end{equation}

Then 
\begin{equation} \label{eq: parametr2}
\begin{split}
\partial_\alpha (r_\alpha^2)& =  2 \frac{|k+k_3|^2}{\dfrac{\Omega'(|k+k_3-z_\alpha|)}{|k+k_3-z_\alpha|} + \dfrac{\Omega'(|z_\alpha|)}{|z_\alpha| } }\dfrac{\Omega'(|k+k_3-z_\alpha|)}{|k+k_3-z_\alpha|} - 2\alpha |k+k_3|^2 \\
&= 2 |k+k_3|^2 \left(   \frac{ ( 1-\alpha) \dfrac{\Omega'(|k+k_3-z_\alpha|)}{|k+k_3-z_\alpha|} - \alpha \dfrac{\Omega'(|z_\alpha|)}{|z_\alpha| }  }{  \dfrac{\Omega'(|k+k_3-z_\alpha|)}{|k+k_3-z_\alpha|} + \dfrac{\Omega'(|z_\alpha|)}{|z_\alpha| }  } \right). 
\end{split}
\end{equation}
Now concerning the $|\nabla_z \mathcal{C}(z)|^2$:
\begin{equation} \label{eq: parametr3}
\begin{split}
&|\nabla_z \mathcal{C}(z)|^2 = \left\vert  \frac{z_\alpha-k-k_3}{|z_\alpha-k-k_3|} \Omega'(|k+k_3-z_\alpha|)+ \frac{z_\alpha}{|z_\alpha|}\Omega'(|z_\alpha|)   \right\vert^2 \\ 
&=  |k+k_3|^2 \left\vert  
(\alpha-1)\dfrac{\Omega'( |k+k_3-z_\alpha| )}{ |k+k_3-z_\alpha| } + \alpha \dfrac{\Omega'(|z_\alpha|)}{|z_\alpha| } \right\vert^2 
 + r_\alpha^2 \left\vert  \dfrac{\Omega'(|k+k_3-z_\alpha|)}{|k+k_3-z_\alpha|} + \dfrac{\Omega'(|z_\alpha|)}{|z_\alpha| }  \right\vert^2.
\end{split}
\end{equation}

Due to cancellations, using \eqref{eq: parametr1}, \eqref{eq: parametr2} and \eqref{eq: parametr3}, we eventually write 
\begin{align}
\label{eq:parametrisation}
\delta_{\omega_1+\omega_2-\omega_3-\omega}  \dd z \dd k_3 
&= \frac{\dd  \mu(z)}{ |\nabla \mathcal{C}(z) | } \dd k_3  
=   \frac{|k+k_3|}{\dfrac{\Omega'(|k+k_3-z_\alpha|)}{|k+k_3-z_\alpha|} + \dfrac{\Omega'(|z_\alpha|)}{|z_\alpha| }} \dd \alpha \dd \theta \dd k_3.
\end{align}
We notice that since $|z_\alpha|^2 = |\alpha|^2|k+k_3|^2+ r_\alpha^2 \leq k_c^2$,  $ |\alpha|^2 = \frac{ |z_\alpha|^2- r_\alpha^2 }{| k+k_3 |^2}$.

Then the nonlinear operator in \eqref{eq:NL WKE} restricted to the domain $[0,k_c]^3$ for each wavenumber, is 
\begin{align*}
& \int_{[0,k_c]^9} \frac{n n_1n_2n_3}{[|k||k_1| |k_2||k_3|]^{\beta/2}} \left(\frac{1}{n_1} + \frac{1}{n_2} - \frac{1}{n_3} - \frac{1}{n}\right) \delta_{\Omega} \delta_{\Sigma} \dd k_1 \dd k_2 \dd k_3 
 \\ & = \int_{[0,k_c]^6} \frac{n(k) n(k+k_3-k_2)n(k_2)n(k_3) }{[|k||k+k_3-k_2||k_2||k_3|]^{\beta/2} } \left(\frac{1}{n(k+k_3-k_2)} + \frac{1}{n(k_2)} - \frac{1}{n(k_3) } - \frac{1}{n(k)}\right) \delta_{\Omega}  \dd k_2 \dd k_3 
 \\ &= \int_{[0,k_c]^3} \Bigg\{ \int_{S_{k,k_3}} \frac{n(k) n(k+k_3-z)n(z)n(k_3)}{ [|k||k+k_3-z| |z||k_3|]^{\beta/2} } \times 
 \\ & \times  \left(\frac{1}{n(k+k_3-z)} + \frac{1}{n(z)} - \frac{1}{n(k_3) } - \frac{1}{n(k)}\right)   
 \frac{\dd \mu(z)}{ |\nabla \mathcal{C}(z) | }  \Bigg\}  \dd k_3
 \\&= 
 \int_{[0,k_c]^3} \Bigg\{ 
 \int_{-\frac{ \sqrt{k_c^2- r_\alpha^2} }{| k+k_3 | }}^{\frac{ \sqrt{k_c^2- r_\alpha^2} }{| k+k_3 | }} \int_0^{2\pi} 
 \frac{n(k) n(k+k_3-z_\alpha)n(z_\alpha)n(k_3)}{[|k||k+k_3-z_\alpha| |z_\alpha||k_3|]^{\beta/2} } \times \\ & \times  \left(\frac{1}{n(k+k_3-z_\alpha)} + \frac{1}{n(z_\alpha)} - \frac{1}{n(k_3) } - \frac{1}{n(k)}\right)   
  \frac{|k+k_3|}{\dfrac{\Omega'(|k+k_3-z_\alpha|)}{|k+k_3-z_\alpha|} + \dfrac{\Omega'(|z_\alpha|)}{|z_\alpha| }} 
 \dd \theta  \dd \alpha  \Bigg\}  
  \dd k_3.
 \end{align*}



\subsection{Splitting of the operator}
The proof of the main results follows techniques already applied for the theory of the Boltzmann equation \cite{GradUnesco, LaureSR09}. The idea is to decompose the operator between a part satisfying the coercivity estimate and a compact part.  We split the operator as follows:
\begin{equation} \label{eq: K and A operators}
\begin{split}
&L_{k_c} g (k) := \int_{[0,k_c]^9 } f_{\infty}(k)^{-1} \int_{\mathbb{R}^9} \frac{ f_{\infty}(k)f_{\infty}(k_1)f_{\infty}(k_2)f_{\infty}(k_3)}{|k|^{\beta/2}|k_1|^{\beta/2} |k_2|^{\beta/2}|k_3|^{\beta/2}} \times \\ 
&\times \big[ f_{\infty}(k_1)^{-1} g_1 + f_{\infty}(k_2)^{-1} g_2 - f_{\infty}(k_3)^{-1} g_3 - f_{\infty}(k)^{-1}  g \big] 
 \delta_{\Omega} \delta_{\Sigma} \dd k_1 \dd k_2 \dd k_3 \\ 
&= \int \frac{f_{\infty}(k_1)f_{\infty}(k_2)f_{\infty}(k_3) }{|k|^{\beta/2}|k_1|^{\beta/2} |k_2|^{\beta/2}|k_3|^{\beta/2}} \big[ f_{\infty}(k_1)^{-1} g_1  +f_{\infty}(k_2)^{-1} g_2  -  f_{\infty}(k_3)^{-1} g_3 \big]
\delta_{\Omega} \delta_{\Sigma} \dd k_1 \dd k_2 \dd k_3   \\
&- \int \frac{ f_{\infty}(k_1)f_{\infty}(k_2)f_{\infty}(k_3) f_{\infty}(k)^{-1} }{ |k|^{\beta/2}|k_1|^{\beta/2} |k_2|^{\beta/2}|k_3|^{\beta/2}} g(k) 
 \delta_{\Omega} \delta_{\Sigma} \dd k_1 \dd  k_2 \dd k_3 \\ 
 &:= \mathcal{K}^{k_c} g(k) - \mathcal{A}^{k_c} g(k).
\end{split}
\end{equation}

Concerning the range of values of $\alpha$ we see that since by the assumption $|z_\alpha|\leq k_c$, $\alpha \in \left[ - \frac{\sqrt{k_c^2-r_\alpha^2}}{|k+k_3|}, \frac{\sqrt{k_c^2-r_\alpha^2}}{|k+k_3|} \right]$. 
Then the proof will follow from the next two lemmas.  
\begin{lemma} \label{lem: coercive A}
Let $\beta \in \R$, the multiplication operator $\mathcal{A}^{k_c}$ defined in \eqref{eq: K and A operators} is \emph{coercive} when $k , k_i \in [0,k_c]^3$ and under hypothesis \ref{ass: dispersion} (i),(ii). In particular 
$$ \int_{ [0,k_c]^3}(\mathcal{A}^{k_c}g(k) )g(k)  \dd k \geq  \frac{ \pi k_c^{-2\beta+1} }{M^3} \int_{[0,k_c]^3}  g(k)^2   \dd k$$
where $M= M(k_c)$ is so that $ f_{\infty}(k)^{-1} \leq M$ for all $k \in [0,k_c]^3$. 
\end{lemma}
\begin{proof} 
 Due to the parametrisation above, we have 
\begin{align*}
&\int_{ [0,k_c]^3}(\mathcal{A}^{k_c}g(k) ) g(k)  \dd k =  \\
& \int_{ [0,k_c]^6}\int_{ [-\frac{k_c}{|k+k_3|},\frac{k_c}{|k+k_3|}]}\int_{0}^{2\pi} 
\frac{f_{\infty}(k+k_3-z_\alpha) f_{\infty}(z_\alpha)f_{\infty}(k_3)}{|k|^{\beta/2}|k_1|^{\beta/2} |k_2|^{\beta/2}|k_3|^{\beta/2}}
 g(k)^2  f_{\infty}(k)^{-1} \times \\ 
& \qquad \times \left( \frac{|k+k_3|}{\dfrac{\Omega'(|k+k_3-z_{\alpha}|)}{|k+k_3-z_{\alpha}|} + \dfrac{\Omega'(|z_{\alpha}|)}{|z_{\alpha}| }}\right)\dd \theta  \dd \alpha  \dd k_3  \dd k 
\gtrsim \\ &C_1 
 \int_{ [0,k_c]^6} \int_{ [-\frac{k_c}{|k+k_3|},\frac{k_c}{|k+k_3|}]}\int_{0}^{2\pi}
   \frac{f_{\infty}(k+k_3-z_\alpha) f_{\infty}(z_\alpha)f_{\infty}(k_3)}{|k|^{\beta/2}|k_1|^{\beta/2} |k_2|^{\beta/2}|k_3|^{\beta/2}}
 \frac{g(k)^2}{f_{\infty}(k)} |k+k_3| \dd \theta  \dd \alpha  \dd k_3  \dd k 
\end{align*}
where we used that $\frac{|k+k_3-z|}{\Omega'(|k+k_3-z|)} \gtrsim \mathcal{O}(1)$. Also if $M = M(k_c)$ is the upper bound on $f_{\infty}(k)^{-1}$, $1+\omega(k)+1 \cdot k \leq M$, we write 
\begin{align*}
&\int_{ [0,k_c]^3}(\mathcal{A}^{k_c}g(k) ) g(k)  \dd k \gtrsim 
\\ & 
\frac{2 \pi k_c^{-2\beta} C_1}{4 M^3 } \int_{[0,k_c]^3}  g(k)^2   \int_{ [0,k_c]^3} \int_{-\frac{k_c}{|k+k_3|}}^{\frac{k_c}{|k+k_3|}} |k+k_3| \dd \alpha \dd k_3  \dd k \gtrsim \\
 &   \frac{2 \pi C_1}{4M^3}2 k_c^{-2\beta+1} \int_{ [0,k_c]^3}  g(k)^2  \dd k =\frac{ \pi k_c^{-2\beta+1} }{M^3} \int_{ [0,k_c]^3}  g(k)^2  \dd k. 
\end{align*}
\end{proof} 

\begin{lemma} \label{lem: compactness K}
The operator $\mathcal{K}^{k_c}$ defined in \eqref{eq: K and A operators} is \emph{compact} when $k, k_i \in [0,k_c]^3$, $\beta=0$ and when hypothesis \ref{ass: dispersion} (i),(ii) holds. 
\end{lemma}
\begin{proof} 
For the \emph{compactness}, we write $$\mathcal{K}^{k_c} g(k) = \mathcal{K}_1^{k_c} g - \mathcal{K}_2^{k_c} g$$ where 
\begin{align*}
\mathcal{K}^{k_c}_2 g &:= \int_{[0,k_c]^3} \int_{-\frac{k_c}{|k+k_3|}}^{\frac{k_c}{|k+k_3|}}\int_{0}^{2\pi} \frac{f_{\infty}(k+k_3-z_\alpha) f_{\infty}(z_\alpha)  g(k_3) }{[|k||k+k_3-z_\alpha| |z_\alpha||k_3|]^{\beta/2}}  \frac{|k+k_3-z_{\alpha}|}{\dfrac{\Omega'(|k+k_3-z_{\alpha}|)}{|k+k_3-z_{\alpha}|} + \dfrac{\Omega'(|z_{\alpha}|)}{|z_{\alpha}| }}  \dd\theta\dd\alpha \dd k_3 
\end{align*}
and
\begin{align*}
\mathcal{K}^{k_c}_1 g(k) :=  \int_{[0,k_c]^3}&\int_{-\frac{k_c}{|k+k_3|}}^{\frac{k_c}{|k+k_3|}} \int_{0}^{2\pi} 
\bigg(\frac{g(k_3+k-z_{\alpha}) f_{\infty}(z_\alpha)f_{\infty}(k_3) }{[|k||k+k_3-z_\alpha| |z_\alpha||k_3|]^{\beta/2}} + \\ & + 
 \frac{g(z_{\alpha}) f_{\infty}(k+k_3-z_\alpha) f_{\infty}(k_3)}{[|k||k+k_3-z_\alpha| |z_\alpha||k_3|]^{\beta/2}} \bigg) 
 \frac{|k+k_3-z_{\alpha}|}{\dfrac{\Omega'(|k+k_3-z_{\alpha}|)}{|k+k_3-z_{\alpha}|} + \dfrac{\Omega'(|z_{\alpha}|)}{|z_{\alpha}| }}\
 \dd\theta\dd\alpha \dd k_3.
\end{align*}
In order to conclude the compactness of both operators $\mathcal{K}^{k_c}_1$, $\mathcal{K}^{k_c}_2$, we identify their integral kernels and show that their kernels are square integrable. Then as Hilbert–Schmidt integral operators they are compact. 

\noindent
\emph{Kernel of $\mathcal{K}_2^{k_c}$}: First the kernel of $\mathcal{K}_2^{k_c}$ is directly seen to be 
$$ k_2^{k_c} (k,k_3) =\int_{\alpha, \theta}  \frac{|k+k_3-z_{\alpha}|}{\dfrac{\Omega'(|k+k_3-z_{\alpha}|)}{|k+k_3-z_{\alpha}|} + \dfrac{\Omega'(|z_{\alpha}|)}{|z_{\alpha}| }}  \frac{f_{\infty}(k+k_3-z_\alpha) f_{\infty}(z_\alpha)}{[|k||k+k_3-z_\alpha| |z_\alpha||k_3|]^{\beta/2}}  \dd \alpha \dd \theta.
$$
We consider the $L^2$-norm of $k_2^{k_c}$ and we write 
\begin{align*}
    \int_{[0,k_c]^3} |k_2^{k_c} (k,k_3)|^2 \dd k_3 \lesssim 
\int_{[0,k_c]^3} & \int_{[-\frac{k_c}{|k+k_3|}, \frac{k_c}{|k+k_3|}] } \int_{0}^{2\pi}  \frac{|k+k_3-z_{\alpha}|^2}{\left(\dfrac{\Omega'(|k_3-z_{\alpha}|)}{|k+k_3-z_{\alpha}|} + \dfrac{\Omega'(|z_{\alpha}|)}{|z_{\alpha}| } \right)^2}  \times \\ & 
\times  \frac{[f_{\infty}(k+k_3-z_\alpha) f_{\infty}(z_\alpha) ]^2}{ [|k||k+k_3-z_\alpha| |z_\alpha||k_3|]^{\beta}}  \dd \alpha \dd  \theta \dd k_3
\end{align*} 
Using now that $\frac{\Omega'(x)}{x} \gtrsim \mathcal{O}(1)$ due to Assumption \ref{ass: dispersion}(ii), this is upper bounded by 
\begin{align*}
& \| f_\infty \|_{\infty}^4 \int_{[0,k_c]^3} \int_{[-\frac{k_c}{|k+k_3|}, \frac{k_c}{|k+k_3|}] } 
\int_{0}^{2\pi} \frac{|k+k_3-z_{\alpha}|^2}{[|k||k+k_3-z_\alpha| |z_\alpha||k_3|]^{\beta}}  \dd \alpha \dd  \theta \dd k_3 = \\ & 
 \| f_\infty \|_{\infty}^4   \int_{[0,k_c]^3} \int_{[-\frac{k_c}{|k+k_3|}, \frac{k_c}{|k+k_3|}] }\int_{0}^{2\pi} \mathbbm{1}_{|k+k_3|\geq \delta} \frac{ |k+k_3-z_{\alpha}|^2}{[|k||k+k_3-z_\alpha| |z_\alpha||k_3|]^{\beta}}  \dd \alpha \dd  \theta \dd k_3 + \\  &
  \| f_\infty \|_{\infty}^4  \int_{[0,k_c]^3} \int_{[-\frac{k_c}{|k+k_3|}, \frac{k_c}{|k+k_3|}] }\int_{0}^{2\pi} \mathbbm{1}_{|k+k_3|\leq \delta} \frac{|k+k_3-z_{\alpha}|^2 }{[|k||k+k_3-z_\alpha| |z_\alpha||k_3|]^{\beta}} \dd \alpha \dd  \theta \dd k_3 = I^\delta_1 + I^\delta_2
\end{align*}
for some positive $0< \delta$. We see that the first term $ I^\delta_1$ is bounded by $ 4\| f_\infty \|_{\infty}^4 \pi k_c k_c^{4}\delta^{-4\beta-1}$ since we are away from $0$. 
For the second term  $I^\delta_2$ we notice that since all the frequencies are positive, $|k+k_3| \leq \delta$ implies that  $|k+ k_3-z_\alpha | < \delta$, thus $I^\delta_2$ is finite as $\delta$ goes to $0$ as long as $\beta \in [0,1/4]$.  



\noindent
\emph{Kernel of $\mathcal{K}_1^{k_c}$}:
Regarding the kernel of $\mathcal{K}_1^{k_c}$, after a change of variables we have 
\begin{align*}
  \mathcal{K}_1^{k_c}g(k) &=  \int_{[0,k_c]^3}\int_{\alpha, \theta} \frac{f_{\infty}(z_\alpha)f_{\infty}(k_3+k-z_\alpha) }{[|k||k+k_3-z_\alpha| |z_\alpha||k_3|]^{\beta/2}} g(z_{\alpha}) \times \\
  & \times 
  \left(
  \frac{|z_{\alpha}|}{\dfrac{\Omega'(|k+ k_3-z_{\alpha}|)}{|k+k_3-z_{\alpha}|} + \dfrac{\Omega'(|z_{\alpha}|)}{|z_{\alpha}| }}
  +
  \frac{|k+k_3-z_{\alpha}|}{\dfrac{\Omega'(|k+ k_3-z_{\alpha}|)}{|k+k_3-z_{\alpha}|} + \dfrac{\Omega'(|z_{\alpha}|)}{|z_{\alpha}| }}
  \right) \dd \alpha \dd \theta \dd k_3 \\ & = 
   \int_{[0,k_c]^6}  g(z) \frac{f_{\infty}(z)f_{\infty}(k_3+k-z)}{[|k||k+k_3-z| |z||k_3|]^{\beta/2}}\frac{|z|+ |k+k_3-z| }{|k+k_3|}  \delta_{\Omega} \dd z\dd k_3.
\end{align*}
So the kernel $k_1^{k_c} $ is by inspection 
$$ k_1^{k_c}(k,z)  =\int_{[0,k_c]^3} \frac{f_{\infty}(z)f_{\infty}(k_3+k-z)}{[|k||k+k_3-z| |z||k_3|]^{\beta/2}} \frac{|z|+ |k+k_3-z| }{|k+k_3|}  \delta_{\Omega}  \dd k_3.$$
We then compute 
\begin{align*} 
& \int_{[0,k_c]^3} |k_1^{k_c} (k,z)|^2 \dd z \lesssim \\
  & \int_{[0,k_c]^3}\int_\alpha \int_\theta \frac{f_{\infty}(z)^2f_{\infty}(k_3+k-z)^2}{[|k||k+k_3-z| |z||k_3|]^{\beta}} \frac{[|z|+ |k+k_3-z|]^2}{|k+k_3|}\frac{\dd \alpha \dd \theta}{ \dfrac{\Omega'(|k+k_3-z_{\alpha}|)}{|k+k_3-z_{\alpha}|} + \dfrac{\Omega'(|z_{\alpha}|)}{|z_{\alpha}| } } \dd k_3 \lesssim \\ & 
  \int_{[0,k_c]^3}\int_\alpha \int_\theta \frac{[|z|+ |k+k_3-z|]^2}{[|k||k+k_3-z| |z||k_3|]^{\beta} |k+k_3| } \dd \alpha \dd \theta  \dd k_3.
 \end{align*}
and direct calculations show that when $\beta=0$, this is finite, due to the assumptions on the dispersion $\Omega$.
This implies that the operator $\mathcal{K}_1^{k_c}$ is compact (as a Hilbert Schmidt operator). 
\end{proof}
\begin{remark} 
Notice that for the compactness we need to take the exponent of the cross-section term to be $\beta = 0$ as otherwise we can not control the singularity that appears in the kernel. If we let $\beta \in \R$, we would have to cut-off further all the frequencies away from $0$. 
\end{remark}
\begin{proof}[Proof of Theorem \ref{theo:coerc linearised}]
An application then of Weyl's theorem yields directly that the essential spectrum of the operator $L_{k_c}$ is the essential spectrum of $\mathcal{A}^{k_c}$. The coercivity of $\mathcal{A}^{k_c}$ implies therefore the existence of a spectral gap of $L_{k_c}$. 
\end{proof}

\section{From the linearized to the non-linear operator}

In this section we aim at proving that the rate of approach to Rayleigh-Jeans equilibria for the nonlinear operator $\mathcal{L}$ is governed by the relaxation rate of the linearized operator, when the dispersion relation is weakly perturbed around the quadratic one.

We take $\beta=0$ and we first rewrite the evolution problem \eqref{eq:NL WKE} as follows, using that $n(t,k)=(\omega+1 \cdot k + 1)^{-1}(1+g(t,k))$.  
\begin{equation}\label{eq:NL evolution}
\begin{split}
    \partial_t g(t,k) &= -f_{\infty}^{-1}(k)
     \int_{\mathbb{R}^9} \left( n n_2 n_3 + n n_1 n_3 - n_1 n_2 n - n_1 n_2 n_3  \right) \delta_{\Omega} \delta_{\Sigma} \dd k_1 \dd k_2 \dd k_3  \\
 &= -L g(t,k) +  \Gamma(g,g) + Q(g,g,g)
 \end{split}
\end{equation}
where $ \Gamma(g,g)$ is the bilinear operator 
\begin{equation}
\begin{split}
 \Gamma(g,g) &= \int_{\mathbb{R}^9} f_{\infty}(k_1)f_{\infty}(k_2)f_{\infty}(k_3) 
  \Big\{ ( f_{\infty}^{-1}(k_1)-f_{\infty}^{-1}(k) ) g_2g_3 +  ( f_{\infty}^{-1}(k_1)- f_{\infty}^{-1}(k_3) ) g_2g 
  \\& \qquad+
    ( f_{\infty}^{-1}(k_1)+f_{\infty}^{-1}(k_2) ) gg_3+ 
   (f_{\infty}^{-1}(k_2)-f_{\infty}^{-1}(k) ) g_1g_3 +  (f_{\infty}^{-1}(k_2)-f_{\infty}^{-1}(k_3) ) gg_1 
   \\ & \qquad\qquad +
     ( -f_{\infty}^{-1}(k_3)-f_{\infty}^{-1}(k_1) ) g_2g_1\Big\}\delta_{\Omega} \delta_{\Sigma} \dd k_1 \dd k_2 \dd  k_3
 \end{split}
 \end{equation}
 and $Q(g,g,g)$ is the term with the cubic nonlinearity 
 \begin{equation}\label{eq:Q_Qi}
\begin{split}
 Q(g,g,g) &=   \int_{\mathbb{R}^9}  f_{\infty}(k_1)f_{\infty}(k_2)f_{\infty}(k_3)   \big\{   
f_{\infty}^{-1}(k_1) gg_2g_3 + f_{\infty}^{-1}(k_2) gg_1g_3 \\ & -f_{\infty}^{-1}(k_3) gg_1g_2 - f_{\infty}^{-1}(k) g_1g_2g_3 
 \big\} 
 \delta_{\Omega} \delta_{\Sigma} \dd k_1 \dd k_2 \dd k_3 \\ 
 &=:Q_1^{+}(g,g,g) + Q_2^{+}(g,g,g) - Q_1^{-}(g,g,g) - Q_2^{-}(g,g,g).
  \end{split}
 \end{equation}
 We rewrite the quadratic term as 
 \begin{align*}
 \Gamma&(g,g) = f_{\infty}(k_1)f_{\infty}(k_2)f_{\infty}(k_3)\int_{\mathbb{R}^9} \big\{ ( f_{\infty}^{-1}(k_1)- f_{\infty}^{-1}(k_3)) (gg_2 - g_1g_3) \\ & 
 + ( f_{\infty}^{-1}(k_1)+ f_{\infty}^{-1}(k_2))(g g_3 -g_1g_2 )  +
  ( f_{\infty}^{-1}(k_2)- f_{\infty}^{-1}(k_3))(g g_1-g_2g_3 )\big\}\delta_{\Omega} \delta_{\Sigma} \dd k_1 \dd k_2 \dd  k_3 \\ & = 
\Gamma^1(g,g) + \Gamma^2(g,g) +\Gamma^3(g,g). 
\end{align*}  
 We need estimates on the $L^2$ - norm of the quadratic and cubic nonlinearities on the collision kernel in terms of the $L^2$ - norm of $g$. Estimates on weighted $L^2$ spaces have been done already in \cite{GIT20} without cut-off on the frequencies, thus here we sketch the proof in our notation for the readers convenience. We adapt slightly the proof to include a bounded perturbation around the quadratic dispersion relation. 
   
 \begin{lemma} \label{lem:NL estimate}
 Let $\beta =0$ and $\omega(k)=|k|^2 + s(|k|)$ satisfying assumption \ref{ass: dispersion} (i),(ii),(iii). We have 
 \begin{equation}
 \| \Gamma(g,g)  + Q(g,g,g)\|_{L^2([0,k_c]^3)} \lesssim C_1 ( \|g \|_{L^2([0,k_c]^3)}^2 + \|g \|_{L^2([0,k_c]^3)}^3) .
 \end{equation}
 \end{lemma}
 
 \begin{proof}
The quadratic nonlinearity is handled easily as in the case of the Boltzmann collision operator. 
  For the cubic nonlinearity we want to show that $\|Q(g,g,g)\|_{L^2} \lesssim \|g\|_{L^2}^3$. 
 
 We estimate first the term $$ Q_2^{-}(g,g,g) = -\int_{[0,k_c]^9}  \frac{f_{\infty}(k_1)f_{\infty}(k_2)f_{\infty}(k_3)}{[|k||k_1||k_2||k_3|]^{\beta/2}} g(k_1) g(k_2) g(k_3) \delta_{\Omega}\delta_{\Sigma} \dd k_1 \dd k_2 \dd k_3 $$ as follows by applying a $TT^*$-type argument:
 \begin{equation}
\begin{split}
&| Q_2^{-}(g,g,g) |  \lesssim\\& \left\vert \int_{\R^3 \times \R} \int_{[0,k_c]^9}g(k_1) g(k_2) g(k_3) e^{i \mu \cdot (-k_1-k_2+k+k_3)} e^{i \nu (-\omega_1-\omega_2+\omega+\omega_3)}   \dd k_1 \dd k_2 \dd k_3 \dd \mu \dd \nu \right\vert  \lesssim \\
& C  \Big\vert \int_{\R^3 \times \R} 
 \left(\int_{[0,k_c]^3} g(k_1) e^{-i\mu \cdot k_1-i\nu\omega_1 }  \dd k_1 \right)
 \left(\int_{[0,k_c]^3} g(k_2) e^{-i\mu \cdot k_2-i\nu\omega_2 }  \dd k_2 \right) \times \\ &\qquad\qquad\qquad\qquad
 \times \left(\int_{[0,k_c]^3} g(k_3) e^{i\mu \cdot k_3+i\nu\omega_3 }  \dd k_3 \right) e^{i\mu \cdot k+i\nu\omega} \dd \mu \dd \nu  \Big\vert 
 \end{split}
\end{equation}
so that for a parameter $s > 1/2 $ and denoting by $T g(\mu,\nu):= \int g(x) e^{i \mu \cdot x+ i\nu \omega(x)} \dd x$ we write: 
 \begin{align*}
\| Q_2^{-}(g,g,g)  &\|_{L^2([0,k_c]^3)} \lesssim \|    Q_2^{-}(g,g,g) \|_{L^2_s} \lesssim C \left\| \int_{\R^3\times \R} Tg(\mu,\nu)^2\ \overline{T\bar{g}} (\mu,\nu)e^{i\mu \cdot k+i\nu\omega} \langle k \rangle^s \dd \mu \dd \nu   \right\|_{L^2_s} \\
&\lesssim \left\| \int_{\R} \mathcal{F}^{-1}[ Tg(k,\nu)^2\ \overline{T\bar{g}} (k,\nu) ]e^{i\nu\omega} \langle k \rangle^s \dd \nu   \right\|_{L^2_s}
\lesssim \int_{\R} \| Tg(\cdot,\nu)^2\ \overline{T\bar{g}} (\cdot,\nu) \|_{H^s} \dd\nu
 \end{align*} 
 where for the last inequality we used that $k$ lives in a bounded domain.  Then this is upper bounded by 
 \begin{equation}
\begin{split}
 \int_{\R} \| Tg(\cdot,\nu)\|_{L^{\infty}}^2 \| T g (\cdot,\nu) \|_{H^s} \dd\nu & \lesssim \sup_{\nu \in \R} \| T g (\cdot,\nu) \|_{H^s}  \int_{\nu} \| Tg(\cdot,\nu)\|_{L^{\infty}}^2 \dd \nu \\ & \lesssim \| g\|_{L_s^2}  \int_{\R} \| Tg(\cdot,\nu)\|_{L^{\infty}}^2 \dd\nu.
 \end{split}
\end{equation}
It remains therefore to provide the bound 
$$ \int_{\R} \| Tg(\cdot,\nu)\|_{L^{\infty}}^2 \dd\nu  \lesssim \|g\|_{L_s^2}^2. $$
It is enough to show that 
(assuming without loss of generality that $g(k) = \langle k \rangle^{-s} h(k)$ for $\|h\|_{L^2}=1$) for $F \in L_{\mu}^1L_{\nu}^2$, $$\left\vert  \int_{\R\times\R^3} Tg(\mu,\nu)F(\mu,\nu) \dd\mu \dd\nu \right\vert \lesssim_s 1$$ and since the left-hand side is bounded by
\begin{align*}
\left\vert  \int_{\R\times\R^3} Tg(\mu,\nu)F(\mu,\nu) \dd\mu \dd\nu \right\vert \lesssim \left\| \langle k \rangle^{-s}\int_{\R\times\R^3}   F(\mu,\nu)e^{-i\mu\cdot k-i\nu\omega(k)} \dd\mu \dd\nu \right\|_{L^2},
\end{align*}
equivalently we shall provide the bound 
$$ \left\vert \int_{\R^3} \int_{(\R^3\times \R)^2} \langle k \rangle^{-2s} F(\mu,\nu)F(\mu',\nu') e^{-ik\cdot\mu - i\nu\omega(k)}e^{ik\cdot\mu' + i\nu'\omega(k)}\dd\mu \dd\nu\dd\mu' \dd\nu' \dd k \right\vert \lesssim 1. $$ 

Now if we were to bound $ \int \langle k \rangle^{-2s}e^{-ik\cdot(\mu-\mu') - i(\nu-\nu')\omega(k)} \dd k$  uniformly in $\mu-\mu' \in \R^3$ and so that it is integrable in $\nu-\nu'$, we can write the left-hand side as 
$$ \int_{\R^2} r(\nu-\nu')  \| F(\cdot, \nu)\|_{L^1}\| F(\cdot, \nu')\|_{L^1}  \dd \nu \dd\nu' . $$
Due to properties of the Schr\"{o}dinger kernel, when $\omega(k)=|k|^2+ s(|k|)$,  taking $s>1/2$ we can find a $\theta = \theta(s) >0$ so that $r(\nu-\nu') \sim  e^{-ic(\nu - \nu' )}\frac{|\nu-\nu'|^{-1+\theta}}{(1+|\nu-\nu'|^2 )^{\theta}} $, thus under the assumption that $\|F(\cdot, \nu)\|_{L_\mu^1} \lesssim 1$, we have the desired bound for $Q_2^{-}$: $\| Q_2^{-}(g,g,g)  \|_{L^2([0,k_c]^3)}$ is upper bounded by $\| g\|_{L^2([0,k_c]^3)}^3$ since the weight is bigger than $1$ and bounded.


Next, for $Q_1^{-}(g,g,g)$ we have 
 \begin{align*} 
|Q_1^{-}(g,g,g)|& := \left\vert \int_{[0,k_c]^9} \frac{f_{\infty}(k_1)f_{\infty}(k_2)f_{\infty}(k_3)}{[|k||k_1||k_2||k_3|]^{\beta/2}}      \frac{g(k)g(k_1)g(k_2)}{f_{\infty}(k_3)}   \delta_{\Omega} \delta_{\Sigma} \dd k_1 \dd k_2 \dd k_3 \right\vert  \\ & = 
\left\vert  \int_{[0,k_c]^6} \frac{ f_{\infty}(k_1)f_{\infty}(k_2)f_{\infty}(k_1+k_2-k)}{[|k||k_1||k_2||k_1+k_2-k|]^{\beta/2}}    \frac{g(k)g(k_1)g(k_2) }{f_{\infty}(k_1+k_2-k)}   \delta_{\Omega} \dd k_1 \dd k_2  \right\vert \\ &  \lesssim  \left\vert  g(k) \int_{[0,k_c]^6} g(k_1) g(k_2)  \delta_{\Omega}
\dd k_2 \dd k_3  \right\vert. 
\end{align*}
We mollify the dirac $\delta_{\Omega}$ in the integrand with a smooth function $\delta^\varepsilon_{\Omega} \in C_c^\infty([-1,1])$ of mass $1$. So for any $\varepsilon \in (0,1]$, 
it is enough to show that 
\begin{align} \label{eq:T_2}
 \left\vert \int_{[0,k_c]^6} g(k_1)g(k_2)  \delta_{\omega(k_1)+\omega(k_2)-\omega(k)-\omega(k_1+k_2-k)}^{(\varepsilon)}
\dd k_1 \dd k_2 \right\vert  \lesssim \| g\|_{L^2([0,k_c]^3)}  \| g\|_{L^2([0,k_c]^3)}. 
\end{align}

We Taylor expand and apply mean-value theorem in order to write for $\theta, \theta' \in (0,1)$
 \begin{align*} 
[ \omega(k)- \omega(k_1) ] &+ [ \omega(k_1+k_2-k) -\omega(k_2)] \\ & =
 \nabla\omega(k_1) \cdot(k-k_1) + (k-k_1)^T \nabla^2 \omega(\theta k +(1-\theta)k_1)) (k-k_1)   \\ & + \nabla \omega(k_2) \cdot (k_1-k) + (k_1-k)^T \nabla^2 \omega(\theta (k_1+k_2-k) +(1-\theta)k_2)) (k_1-k)\\ & = 
 (k-k_1)\cdot[\nabla\omega(k_1) - \nabla\omega(k_2)] 
 \\ & + (k-k_1)^T [\nabla^2 \omega(\theta (k_1+k_2-k) +(1-\theta)k_2))+  \nabla^2 \omega(\theta' k +(1-\theta')k_1)) ](k-k_1).
 \end{align*}

Assuming that $\omega(x) = |x|^2+ s(|x|)$, with $s(|x|)=S(x)$, and  $ \|s\|_{L^{\infty}} \leq C_1$ this equals to for $\theta, \theta' \in (0,1) $:
\begin{align*} 
2 &(k-k_1) \cdot (k-k_2) + (k-k_1)\cdot \nabla [S (k_1) - S(k_2)] + 
\\ &+ (k-k_1)^T[\nabla^2 S(\theta k +(1-\theta)k_1)+ \nabla^2 S( \theta' (k_1+k_2-k) +(1-\theta')k_2 )](k-k_1).
\end{align*}

The left-hand side of \eqref{eq:T_2} is written as 
\begin{align*} 
 &\Bigg\vert \int_{[0,k_c]^6} g(k_1)g(k_2)
 \delta^{(\varepsilon)} \Big(2 (k-k_1) \cdot (k-k_2) + (k-k_1)\cdot \nabla [S (k_1) - S(k_2)] + \\ & (k-k_1)^T[\nabla^2 S(\theta k +(1-\theta)k_1)+ \nabla^2 S( \theta' (k_1+k_2-k) +(1-\theta')k_2 )](k-k_1) \Big) 
\dd k_1 \dd k_2 \Bigg\vert \\ 
& \lesssim 
\Bigg\vert \int_{[-k_c,k_c]^6} 
 \delta^{(\varepsilon)} \Big(2 k_1 \cdot k_2 + 2 \Lambda_i |k_1|^2 + k_1  \cdot \nabla [S (k-k_1) - S(k-k_2)] 
  \Big) \frac{G(k_1)}{\langle k+k_1 \rangle^s} \frac{G(k_2)}{\langle k+k_2 \rangle^s}
\dd k_1 \dd k_2 \Bigg\vert
\\ 
&= \Bigg\vert 
 \int_{[-k_c,k_c]^6} 
 \delta^{(\varepsilon)} 
 \Big(2 k_1 \cdot k_2 + 2 \Lambda_i |k_1|^2 + k_1  \cdot   \nabla^2 S (\theta'' (k-k_1) + (1-\theta'')(k-k_2)) (k_2-k_1) \Big) \times \\& \qquad\qquad\qquad\qquad\qquad\qquad\qquad\qquad\qquad\qquad\qquad\times  
\frac{G(k_1)}{\langle k+k_1 \rangle^s} \frac{G(k_2)}{\langle k+k_2 \rangle^s} \dd k_1 \dd k_2 \Bigg\vert
\end{align*} 
where we performed a change of variables and where $G(k_1)  = \langle k+k_1 \rangle^sg(k-k_1)$ for some parameter $s>1/2$. Also $\Lambda_i=\Lambda_1$ if we are on the decreasing side of the mollified function and  $\Lambda_2$ if we are on the increasing. Finally we applied again mean-value theorem with interpolation parameter $\theta'' \in (0,1)$.

Now equivalently we want to prove the $L^2$-boundedness of the following linear operator, uniformly in $\varepsilon$ and $k$:
\begin{align*} 
L_2 G(k_1) &: = \int_{[-k_c,k_c]^3} \delta^{(\varepsilon)} 
 \Big(2 k_1 \cdot k_2 + 2 \Lambda_i |k_1|^2 + k_1  \cdot   \nabla^2 S (\theta'' (k-k_1) + (1-\theta'')(k-k_2)) (k_2-k_1) \Big)  \times \\& \qquad\qquad\qquad\qquad\qquad\qquad\qquad\qquad\qquad\qquad\qquad\times \frac{G(k_2)}{\langle k+k_1 \rangle^s \langle k+k_2 \rangle^s}  \dd k_2 \\ 
 & =  \int_{[-k_c,k_c]^3} \delta^{(\varepsilon)} 
 \Big(k_1 \cdot (2+ \nabla^2 S (\xi )k_2 + a \Big) \frac{G(k_2)}{\langle k+k_1 \rangle^s \langle k+k_2 \rangle^s}  \dd k_2,
\end{align*} 
where $a = 2 \Lambda_i |k_1|^2 - k_1 \cdot \nabla^2 S(\xi)k_1$ is independent of $k_2$. 
Since $\nabla^2 S$ is a diagonal matrix with positive eigenvalues, after a change of variables $k_2 \mapsto (2+ \nabla^2 S (\xi ) ) k_2=: Dk_2$, up to some constants we need to bound the operator 
$$ L_2 \bar{G}(k_1) = \int_{k_2 \in C_{k_c}} \delta^{(\varepsilon)} \Big( k_1 \cdot k_2 + a \Big)\frac{\bar{G}(k_2)}{\langle k+k_1 \rangle^s \langle k+D^{-1}k_2 \rangle^s}  \dd k_2
$$
where $C_{k_c}$ is a compact domain and $\bar{G} (k_2) = G(D^{-1}k_2)$. As we can now apply  \cite[Lemma 4.3]{GIT20}, which bounds the integral kernel of such operator, we conclude the boundedness of $L_2$ by a $TT^*$ argument as in \cite[subsection 4.2]{GIT20}.
 


Similar calculations work for the remaining two terms: \\ $Q_1^{+}(g,g,g):= \int_{[0,k_c]^9} f_{\infty}(k_1)f_{\infty}(k_2)f_{\infty}(k_3) \big\{   f_{\infty}(k_1)^{-1} g(k)g(k_2)g(k_3) \big\}  \delta_{\Omega} \delta_{\Sigma} \dd k_1 \dd k_2 \dd k_3$ and 
$Q_2^{+}(g,g,g):= \int_{[0,k_c]^9} f_{\infty}(k_1)f_{\infty}(k_2)f_{\infty}(k_3) \big\{   f_{\infty}(k_2)^{-1} g(k)g(k_1)g(k_3) \big\}  \delta_{\Omega} \delta_{\Sigma} \dd k_1 \dd k_2 \dd k_3$.


  \end{proof}

  \subsection{Existence and uniqueness of solution close to Rayleigh-Jeans} \label{subsect: existence}
  
  In this section we show the existence of solutions close to the Rayleigh-Jeans equilibrium. For such solutions we will then show that they are $L^2$-stable. For the existence we will proceed via a fixed point argument. 
 Let $S^t$ be the semigroup generated by $L_{k_c}$. Then Duhamel's formula reads
 \[
 g(t,k) = S^t g(0,k) + \int_0^t S^{t-s}  \{ \Gamma(g_s,g_s)(k)  + Q(g_s,g_s,g_s)(k) \} \dd s. 
 \]
Motivated by this, we define the following operator   
\[
\mathcal{F}g(k) :=  S^tg(0,k)+ \int_0^t S^{t-s}  \{ \Gamma(g_s,g_s)(k)  + Q(g_s,g_s,g_s)(k) \} \dd s.
 \]
 We also introduce the norm $\|g\| =\sup_{t \geq 0} \|g(t,\cdot) \|_{L^2 ([0,k_c]^3)}$. 
 
 Next we observe that we can decompose the solution $g$ into an element on the subspace of collision invariants and an element vertical to this: $g_t = g_t^1+g_t^2$ with $g_t^1=\Pi g_t$ being the component along $\operatorname{Ker}(L_{k_c})$ which is  constant in time as the kinetic wave equation that we consider is the spatial homogeneous one, and $g_t^2 = \Pi^{\perp}g_t$. Here with $\Pi$ we denote the orthogonal projection on $\operatorname{Ker}(L_{k_c})$. Then 
 $$ g_t^2(k) = S^t g_0^2(k) + \int_0^t  S^{t-s}  \{ \Gamma(g_s,g_s)(k)  + Q(g_s,g_s,g_s)(k) \}  \dd s. $$
 
  Fix a ball $H= \left\{g \in C((0,\infty);L^2([0,k_c]^3)): \|g \|_{L^2([0,k_c]^3)} \leq c(\lambda) \right\}$ for a constant $c(\lambda)$ depending on $\lambda = \lambda(k_c)$ to be determined later. We then want to show that as soon as the initial data satisfy $\|g_0\|_{L^2([0,k_c]^3)} \leq c(\lambda)/2$: 
  \begin{itemize} 
  \item $ \mathcal{F} H \subset H $ and 
  \item the following contraction property for $\mathcal{F}$: when $g, \tilde{g} \in H$, $\|\mathcal{F}g - \mathcal{F} \tilde{g} \| \leq c \| g-\tilde{g}\|$, for $c<1$. 
  \end{itemize}
  
Let $g \in H$. We calculate 
  \[
  \| \mathcal{F} g_2 \|_{L^2} \lesssim \| g_0\|_{L^2} e^{- \lambda t} + \int_0^t e^{- \lambda (t-s)} \| \Gamma(g_s,g_s)  + Q(g_s,g_s,g_s)\| \dd s
  \]
  where the exponential relaxation of the semigroup of the linearized operator $L_{k_c}$ with spectral gap $ \lambda $ is implied by Grönwall's inequality, using the coercivity property, cf Theorem \ref{theo:coerc linearised}. This in particular means that we assume $g_t \in \operatorname{Ker}(L_{k_c})^{\perp}$ for all $t \geq 0$. 
  Then due to Lemma \ref{lem:NL estimate} the integrand is bounded by 
 \begin{align*}
  e^{- \lambda (t-s)} \| \Gamma(g_s,g_s)  + Q(g_s,g_s,g_s)\| & \lesssim C_1 e^{- \lambda (t-s)} \Big( \|g_s \|_{L^2([0,k_c]^3)}^2 + \|g_s \|_{L^2([0,k_c]^3)}^3 \Big) \\ 
  &\lesssim C_1 e^{- \lambda (t-s)} \left( c(\lambda )^2 + c(\lambda)^3 \right).
   \end{align*}
So that 
 \begin{equation}
\begin{split}
   \| \mathcal{F} g_2 \|_{L^2} &\lesssim \| g_0\|_{L^2} e^{- \lambda t} +  C_1( c(\lambda)^2 +  c(\lambda )^3 ) \int_0^t  e^{- \lambda (t-s)} ds \\ & \lesssim \| g_0\|_{L^2} e^{- \lambda t} +  \frac{ c(\lambda)^2+  c(\lambda)^3 }{\lambda} C_1 (1- e^{-\lambda t})   \\ & 
   \lesssim  
   \frac{c(\lambda)}{2} e^{- \lambda  t} + C_1\frac{ c(\lambda)^2+  c(\lambda)^3 }{\lambda }  \lesssim  c(\lambda)
\end{split}
\end{equation}
This is for example satisfied when $c(\lambda) = \frac{\sqrt{\lambda}}{2\sqrt{2C_1}}$, where $\lambda = \lambda(k_c)$.

\noindent 
\emph{Contraction of the semigroup in the ball $H$}: 
Moving now to the contraction property we have 

\begin{align*}
&\| \mathcal{F}(g) - \mathcal{F}(\tilde{g}) \|_{L^2([0,k_c]^3)} \leq \int_0^t  \| S^{t-s}  \{ (\Gamma(g_s,g_s) + Q(g_s,g_s,g_s)) - \\ &\qquad\qquad\qquad\qquad\qquad\qquad\qquad\qquad   (\Gamma(\tilde{g}_s,\tilde{g}_s) + Q(\tilde{g}_s,\tilde{g}_s, \tilde{g}_s)) \} \|_{L^2([0,k_c]^3)} \dd s \\ 
& \lesssim \int_0^t e^{-(t-s)\lambda  } \| (\Gamma(g_s,g_s) + Q(g_s,g_s,g_s)) - (\Gamma(\tilde{g}_s,\tilde{g}_s) + Q(\tilde{g}_s,\tilde{g}_s, \tilde{g}_s) ) \|_{L^2([0,k_c]^3)} \dd s \\ 
& \lesssim \int_0^t e^{-(t-s)\lambda  } [ \| \Gamma(g_s,g_s) - \Gamma(\tilde{g}_s,\tilde{g}_s)   \|_{L^2([0,k_c]^3)}  + \| Q(g_s,g_s,g_s) - Q(\tilde{g}_s,\tilde{g}_s, \tilde{g}_s) \|_{L^2([0,k_c]^3)} ]  \dd s \\ 
& \lesssim  \int_0^t e^{-(t-s)\lambda} \Big[
 \sum_{i \in \{1,2,3\}} \| \Gamma^i( g_s+\tilde{g}_s, g_s- \tilde{g}_s ) \|_{L^2([0,k_c]^3)} + \\ & 
 \| Q^{+}(g_s,g_s,g_s)  - Q^{+}(\tilde{g}_s,\tilde{g}_s, \tilde{g}_s)\|_{L^2([0,k_c]^3)} +
 \| Q^{-}(g_s,g_s,g_s)  - Q^{-}(\tilde{g}_s,\tilde{g}_s, \tilde{g}_s)\|_{L^2([0,k_c]^3)}
 \Big] \dd s
\end{align*}
where $
Q^{+}(g,g,g) = Q_1^{+}(g,g,g) + Q_2^{+}(g,g,g)$ and $Q^{-}(g,g,g) =Q_1^{-}(g,g,g)  + Q_2^{-}(g,g,g)$ 
as defined in \eqref{eq:Q_Qi}. 
The first term in the bracket in the integrand consisting of quadratic nonlinearities, is bounded by
$$ \sum_{i \in \{1,2,3\}} \| \Gamma^i( g_s+\tilde{g}_s, g_s- \tilde{g}_s ) \|_{L^2([0,k_c]^3)} \leq \sum_{i \in \{1,2,3\}} \| g_s+\tilde{g}_s\| \|g_s- \tilde{g}_s  \|. $$
For the second term: 
\begin{align*} 
&\| Q^{+}(g_s,g_s,g_s)  - Q^{+}(\tilde{g}_s,\tilde{g}_s, \tilde{g}_s)\|_{L^2([0,k_c]^3)} =\\
& \Big\| \int_{[0,k_c]^9 } f_{\infty}(k_1)f_{\infty}(k_2)f_{\infty}(k_3) 
\Big[ f_{\infty}(k_1)^{-1} g (g_2-\tilde{g}_2)g_3 + f_{\infty}(k_1)^{-1} (g-\tilde{g}) g_3\tilde{g}_2
 \\ &\qquad\qquad 
 + f_{\infty}(k_1)^{-1} \tilde{g}\tilde{g}_2(g_3-\tilde{g}_3)  +   f_{\infty}(k_2)^{-1}g(g_1-\tilde{g}_1)g_3 \\ 
&\qquad\qquad \qquad +f_{\infty}(k_2)^{-1} (g - \tilde{g})\tilde{g}_1g_3+ \omega_2 \tilde{g} \tilde{g}_1(g_3- \tilde{g}_3)   \Big] \dd k_1 \dd k_2 \dd k_3\Big\|_{L^2([0,k_c]^3)}\\
&:=\| Q_1^{+}(g,g- \tilde{g}, g) +  Q_1^{+}(g- \tilde{g}, \tilde{g}, g) +  Q_1^{+}(\tilde{g},\tilde{g}, g- \tilde{g}) +\\
& \qquad\qquad \qquad \qquad\qquad \qquad  Q_2^{+}(g,g-\tilde{g}, g) + Q_2^{+}(g-\tilde{g},\tilde{g}, g) + Q_2^{+}(\tilde{g},\tilde{g}, g- \tilde{g})\|_{L^2([0,k_c]^3)} \\ 
&\lesssim  \|g-\tilde{g}\| ( \|g\|^2 + \|g\| \| \tilde{g} \|+ \|\tilde{g}\|^2 )
\end{align*}
where we applied Lemma \ref{lem:NL estimate}.
For the third term: 
\begin{align*} 
&\| Q^{-}(g_s,g_s,g_s)  - Q^{-}(\tilde{g}_s,\tilde{g}_s, \tilde{g}_s)\|_{L^2([0,k_c]^3)}  =
\\ &  \Big\| \int_{[0,k_c]^9 }  f_{\infty}(k_1)f_{\infty}(k_2)f_{\infty}(k_3)   \Big[  f_{\infty}(k_3)^{-1} (g-\tilde{g})g_1g_2 + f_{\infty}(k_3)^{-1} g_2\tilde{g}(g_1 - \tilde{g}_1)  +
\\ & \qquad f_{\infty}(k_3)^{-1} \tilde{g}\tilde{g}_1(g_2 - \tilde{g}_2)  + f_{\infty}(k)^{-1} g_1g_2(g_3-\tilde{g}_3)\\ 
& \qquad + 
 f_{\infty}(k)^{-1} g_2(g_1-\tilde{g}_1)\tilde{g}_3+ f_{\infty}(k)^{-1} \tilde{g}_1\tilde{g}_3 (g_2-\tilde{g}_2)  
\Big]  \dd k_1 \dd k_2 \dd k_3 \Big\|_{L^2([0,k_c]^3)} \\
&:=  \| Q^{-}_1(g-\tilde{g},g,g) + Q^{-}_1(\tilde{g}, g-\tilde{g},g) + Q^{-}_1(\tilde{g},\tilde{g},g-\tilde{g}) +\\ 
&\qquad\qquad \qquad \qquad\qquad \quad 
  Q_2^{-}(g,g,g-\tilde{g})+ Q_2^{-}( g-\tilde{g},g,\tilde{g}) +  Q_2^{-}( \tilde{g}, g-\tilde{g},\tilde{g}) \|_{L^2([0,k_c]^3)} \\
  &\lesssim  \|g-\tilde{g}\| (\|g\|^2 + \|g\|\|\tilde{g}\|+\|\tilde{g}\|^2 )
\end{align*}
due to the estimates in Lemma \ref{lem:NL estimate} again. 
That is, since $g,\tilde{g} \in H$,
\begin{equation} 
\begin{split} 
\| \mathcal{F}(g_t) - \mathcal{F}(\tilde{g}_t) \|_{L^2([0,k_c]^3)} \lesssim (c(\lambda)+ c(\lambda)^2) \int_0^t e^{-(t-s)\lambda }  \|g_s-\tilde{g}_s\| \dd s. 
\end{split}
\end{equation}
Finally we take the supremum in time $t$ to have that $$\| | \mathcal{F}(g) - \mathcal{F}(\tilde{g}) | \| \leq \frac{1}{2}  \| g-\tilde{g}\| .$$
Thus in this norm $\mathcal{F}$ is a contraction mapping.  Existence and uniqueness of the solution now immediately follows. 

\subsection{Stability around Rayleigh-Jeans in $L^2([0,k_c]^3)$}
In this subsection we show stability for the homogeneous kinetic wave equation for data close to the Rayleigh-Jeans equilibrium.
First we multiply both sides of the equation \eqref{eq:NL evolution} by $g$ and we integrate over the frequency space. Since $g_t=g_t^1 + g_t^2$ and $\partial_t g_t^1 =0$, 
\begin{equation}
\begin{split}
 \partial_t &  \|g_t^2\|_{L^2([0,k_c]^3)}^{2}  = - \langle L_{k_c} g_t^2, g_t ^2\rangle_{L^2([0,k_c]^3)} +  \langle (\Gamma(g_t,g_t)  + Q(g_t,g_t,g_t) ),g_t \rangle_{L^2([0,k_c]^3)} \\
 & \lesssim 
 - \lambda \|g_t^2 \|_{L^2([0,k_c]^3)}^2 +2C \Big\{ \|g_t \|_{L^2([0,k_c]^3)}^2 + \|g_t \|_{L^2([0,k_c]^3)}^3 \Big\}  \|g_t \|_{L^2([0,k_c]^3)}. 
  \end{split}
\end{equation}
This implies that as long as $\|g_t\|_{L^2([0,k_c]^3)}+ \|g_t\|_{L^2([0,k_c]^3)}^2 < \frac{\lambda}{4C}$, 
\begin{equation}
\begin{split}
\partial_t \| g_t^2 \|_{L^2([0,k_c]^3)} &\leq - \lambda \|g_t^2 \|_{L^2([0,k_c]^3)} + \frac{\lambda}{2}\|g_t\|_{L^2([0,k_c]^3)}\\ &\leq - \lambda \|g_t^2 \|_{L^2([0,k_c]^3)}  +  \frac{\lambda}{2} \|g_t^2 \|_{L^2([0,k_c]^3)} +  \frac{\lambda}{2} \|g_0^1 \|_{L^2([0,k_c]^3)}\\ 
& = -\frac{\lambda}{2} \|g_t^2 \|_{L^2([0,k_c]^3)}+  \frac{\lambda}{2} \|g_0^1 \|_{L^2([0,k_c]^3)}
 \end{split}
\end{equation}
as $g^1$ is parallel to the subspace of collision invariants. This yields, by Grönwall's lemma, exponential relaxation of $g_t^2:= \Pi^{\perp} g_t$ as soon as we assume that $\Pi g_0=0$ (using also conservation of mass).

Notice that without the assumption of having $\Pi g_0=0$, one gets
\begin{equation}
\begin{split}
\| g_t^2 \|_{L^2([0,k_c]^3)}& \leq e^{ -\frac{\lambda}{2}t} \| g_0^2 \|_{L^2([0,k_c]^3)} + (1-e^{-\frac{\lambda}{2}t}) \| g_0^1 \|_{L^2([0,k_c]^3)} 
\end{split}
\end{equation}
This inequality now means that as long as $\| g_0\|_{L^2([0,k_c]^3)}\leq \frac{1}{2}\sqrt{\frac{\lambda}{2C}}$, i.e. $g_0 \in H$ where existence of solutions is guaranteed, the perturbation $g$ around the equilibrium solution $f_{\infty}(k)$ remains small for any time $t$. Thus $n(t,k) = f_{\infty}(k) (1+g(t,k))$ remains in $L^2([0,k_c]^3)$ close to $f_{\infty}(k)$. 

\smallsection{Acknowledgements} I would like to thank Laure Saint-Raymond for many useful discussions and suggestions. Support by a Huawei fellowship at IHES is acknowledged.

\bibliographystyle{alpha}
\bibliography{bibliography}

\end{document}